\documentclass[reqno0,12pt]{amsart}
\pdfoutput=1
\usepackage{dsfont}
\usepackage{bbm}
\usepackage[nodate]{datetime}
\usepackage[hmargin=20mm,top=28mm,bottom=28mm,a4paper]{geometry}
\usepackage{color}
\usepackage{subfig}
\usepackage{fancyhdr}
\usepackage{amsmath,amssymb,amsfonts,amsthm}
\usepackage{amstext}
\usepackage{amsmath}
\usepackage{amssymb}
\usepackage{enumerate}
\usepackage{amsbsy}
\usepackage{amsopn}
\usepackage{bbm,amsthm}
\usepackage{amscd}
\usepackage{amsxtra}
\usepackage{todonotes}

\newtheorem{theorem}{Theorem}[section]
\newtheorem*{theorem*}{Theorem}
\newtheorem{lemma}{Lemma}[section]

\newtheorem{corollary}{Corollary}[section]
\newtheorem{conjecture}{Conjecture}[section]
\newtheorem{definition}{Definition}[section]

\theoremstyle{remark}

\newcommand{\CC}{\mathds{C}}

\newcommand{\NN}{\mathds{N}}

\DeclareMathOperator{\re}{Re}
\DeclareMathOperator{\res}{Res}

\begin{document}
\title{$\Omega$-bounds for the partial sums of some modified Dirichlet characters}
\author{Marco Aymone}
\begin{abstract}
We consider the problem of $\Omega$ bounds for the partial sums of a modified character, \textit{i.e.}, a completely multiplicative function $f$ such that $f(p)=\chi(p)$ for all but a finite number of primes $p$, where $\chi$ is a primitive Dirichlet character. We prove that in some special circumstances, $\sum_{n\leq x}f(n)=\Omega((\log x)^{|S|})$, where $S$ is the set of primes $p$ where $f(p)\neq \chi(p)$. This gives credence to a corrected version of a conjecture of Klurman et al., Trans. Amer. Math. Soc.,  374 (11), 2021, 7967–7990. We also compute the Riesz mean of order $k$ for large $k$ of a modified character, and show that the Diophantine properties of the irrational numbers of the form $\log p / \log q$, for primes $p$ and $q$, give information on these averages. 
\end{abstract}

\maketitle

\section{Introduction.}

A quite simple statement about a completely multiplicative\footnote{A function $f:\NN\to\CC$ is completely multiplicative if $f(nm)=f(n)f(m)$ for all positive integers $n$ and $m$, and hence, such functions are determined by its values at primes.} function, and that was proved only recently, is that no matter how we choose the values $f(p)$ at primes $p$ on the unit circle, we will always end up with a completely multiplicative function whose partial sums are unbounded, \textit{i.e.}, 
$$\limsup_{x\to\infty}\left|\sum_{n\leq x}f(n)\right|=\infty.$$

This result was proved in 2015 by Tao \cite{taodiscrepancy} in the context of the \textit{Erd\H{o}s discrepancy problem}. Dirichlet characters\footnote{Dirichlet characters of modulus $q$, often denoted by $\chi$, are $q$-periodic completely multiplicative functions such that $\chi(n)=0$ whenever $\gcd(n,q)>1$.} play an important role in the Erd\H{o}s discrepancy Theory: The non-principal characters $\chi$ are completely multiplicative and have bounded partial sums. But of course they are not a counterexample to the result of Tao since they vanish at a finite subset of primes. However, they seem to be extremal in discrepancy theory due to the following fact: Up to this date, the known example of a completely multiplicative $f$ with $|f|=1$ whose partial sums have the lowest known fluctuation is obtained by adjusting a non-principal character $\chi$ at the the primes $p$ where $\chi(p)=0$. For instance (see Borwein-Choi-Coons \cite{Borweindicrepancy}) we can take the non-principal character mod $3$, say $\chi_3$, and define $f(3)=\pm 1$ and $f(p)=\chi_3(p)$ for all the other primes. The partial sums up to $x$ of this modification of $\chi_3$ are $\ll \log x$. 

Going further, we define:

\begin{definition}[Modified characters]\label{definicao modified} We say that $f:\NN\to\{z\in\CC:|z|\leq 1\}$ is a modified character if $f$ is completely multiplicative, and if there is a Dirichlet character $\chi$  and a finite subset of primes $S$ such that:
\begin{enumerate}
\item For all primes $p\in S$, $f(p)\neq \chi(p)$.
\item For all primes $p\notin S$, $f(p)=\chi(p)$.
\end{enumerate}
In this case we also say that $f$ is a modification of $\chi$ with modification set $S$.
\end{definition}

One can easily show (see \cite{aymoneresemblingmobius} for a proof using a Tauberian result) that for a \textit{modified character} we have
\begin{equation}\label{equacao cota superior}
\sum_{n\leq x}f(n)\ll (\log x)^{|S|}.
\end{equation}
A very open question concerns $\Omega$ bounds for the partial sums of such modifications.

The first treatment to such an $\Omega$ bound is due to Borwein-Choi-Coons \cite{Borweindicrepancy} where they considered the case $|S|=1$ and showed that for a quadratic Dirichlet character $\chi$ with prime modulus $q$, if we set $f(p)=\chi(p)$ for all primes except at $q$, where $f(q)= 1$, then the partial sums 

$$\sum_{n\leq x}f(n)=\Omega(\log x).$$

Later, Klurman et al. \cite{klurmanteravainen}, among other results, proved a stronger form of Chudakov's conjecture \cite{klurmanchudakov}, and also conjectured that the partial sums of a modified character $f$ are actually $\Omega((\log x)^{|S|})$. Their Corollary 1.6 states that the partial sums are $\Omega(\log x)$ if the set of modifications $S$ is finite and for at least one prime $p\in S$, $|f(p)|=1$. Their proof is correct when the Dirichlet character is primitive, but this statement is not true for all characters. This is because, if we modify a non-primitive character $\chi$ at a prime $r$ where $\chi(r)=0$, then there are circumstances where the modification $f$ might end up in another character with smaller modulus. Therefore, we restate their conjecture in the following form.

\begin{conjecture} Let $f$ be a modification of a primitive Dirichlet character $\chi$. Assume that for each prime $p$ in the set of modifications $S$ we have $|f(p)|=1$. Then  
$$\sum_{n\leq x}f(n)=\Omega((\log x)^{|S|}).$$ 
\end{conjecture} 

The first result of this paper gives credence to this conjecture by showing that it is true in some very special circumstances.

\begin{corollary} Let $f$ be a modification of a primitive Dirichlet character $\chi$ such that $\chi(-1)=-1$. If for each prime $p$ in the set of modifications $S$ we have $f(p)=+1$, then 
$$\sum_{n\leq x}f(n)=\Omega((\log x)^{|S|}).$$ 

\end{corollary}

The result above is a direct consequence of the slightly more general Theorem below.

\begin{theorem}\label{Teorema principal} Let $f$ be a modification of a primitive Dirichlet character $\chi$ such that for each prime $p$ in the set of modifications $S$, $|f(p)|=1$. Let 
$$T=\sum_{\substack{p\in S\\f(p)=1}}1-\sum_{\substack{p\in S\\\chi(p)=1}}1.$$ 
Let 
$$N=\begin{cases}\max\{0,T\}$, \mbox{ if } $\chi(-1)=-1,\\\max\{0,T-1\}, \mbox{ if } \chi(-1)=1.\end{cases}$$
Then
$$\sum_{n\leq x}f(n)=\Omega((\log x)^{N}).$$
\end{theorem}

One of the reasons behind the result above is that the Dirichlet series of $f$, say $F(s)$, can be written as 
\begin{equation}\label{equacao F}
F(s)=\left(\prod_{p\in S}\frac{1-\frac{\chi(p)}{p^s}}{1-\frac{f(p)}{p^s}}\right)L(s,\chi),
\end{equation}
and the proof consists in analyzing the behaviour of $F(s)$ as $s\to 0^+$, and hence, the functional equation for $L(s,\chi)$ is important here.
   
\subsection{Riesz means and irrationality of the numbers $\log p / \log q$}
Another way to measure the mean behaviour of a sequence $(f(n))_n$ is by its Riesz means. The Riesz mean of order $1$
is defined as
$$R_1(x):=\sum_{n\leq x}f(n)\log (x/n).$$
After partial summation it is not difficulty to see that $R_1(x)$ is equal to $\log x$ times the logarithmic average of $\sum_{n\leq x}f(n)$:
$$R_1(x)=\log x \times \frac{1}{\log x}\int_1^x \left(\sum_{n\leq t}f(n)\right)\frac{dt}{t}.$$
 Therefore, it may regarded as a smooth average of $f(n)$. 

Going further, we can define (see the book of Montgomery and Vaughan \cite{montgomerylivro}) the Riesz mean of order $k$:

$$R_k(x):=\frac{1}{k!}\sum_{n\leq x}f(n)(\log (x/n))^k.$$
It is not difficult to see that $\Omega$ results for $R_k(x)$ transfer to the classical partial sums, as upper bounds for the classical partial sums transfer to $R_k(x)$. 

In our next result we were able to compute with precision the Riesz mean of order $k$ of a modified character, which allows us to say that the exponent $N$ in Theorem \ref{Teorema principal} is optimal in the Riesz mean context.

\begin{theorem}\label{Teorema riesz} Let $f$ be a modification of a primitive Dirichlet character $\chi$ such that for each prime $p$ in the set of modifications $S$, $f(p)$ is a root of unity. Let $T$ be as in Theorem \ref{Teorema principal} and 
$$M=\begin{cases}T,& \mbox{ if }  \chi(-1)=-1,\\ T-1,& \mbox{ if } \chi(-1)=1.\end{cases}$$ 
Then there exists a constant $\gamma>0$ such that for all integers $k\geq 10+\gamma(|S|+1)\max_{p\in S}(\log p)^2$, as $x\to\infty$
$$\sum_{n\leq x}f(n)(\log x/n)^k=\begin{cases}&P(\log x)+O(1),\mbox{ if }M+k\geq 1,\\&O(1), \mbox{ if }M+k\leq 0,\end{cases}$$
where in the case that $M+k\geq 1$, $P(x)$ is a polynomial of degree $M+k$ with leading coefficient $a_{M+k}$ given by
$$a_{M+k}=c_\chi\frac{k!}{(M+k)!}\left(\prod_{\substack{p\in S\\ f(p)=1}}\frac{1-\chi(p)}{\log p}\right)\left(\prod_{\substack{p\in S\\ \chi(p)=1}}\frac{\log p}{1-f(p)}\right)\left(\prod_{\substack{p\in S\\ f(p),\,\chi(p)\neq 1}}\frac{1-\chi(p)}{1-f(p)}\right),$$
where 
$$c_\chi=\begin{cases}L(0,\chi),\mbox{ if } \chi(-1)=-1,\\ L'(0,\chi),\mbox{ if }\chi(-1)=1.\end{cases}$$
\end{theorem}

It is not difficult to see that if $\sum_{n\leq x}a_n\sim (\log x)^N$, then the Riesz mean of order $k$ of this sequence is asymptotically equal to a polynomial $P(\log x)$ of degree $N+k$. Therefore, the result above says that Theorem \ref{Teorema principal} is optimal in the Riesz mean context.

From an analytic point of view, the Riesz mean behaves more or less as the C\'esaro mean, and in the context of modified characters, Duke and Nguyen \cite{duke_riesz} considered the content of Theorem \ref{Teorema riesz} in the case that $|S|=1$. Their results allowed them to state that there exists a $\pm 1$ completely multiplicative function with bounded C\'esaro ``\textit{discrepancy}'', in contrast with the infinite discrepancy of Tao \cite{taodiscrepancy}. Our result above treats the case $|S|>1$. The new problem that didn't appear in the case $|S|=1$, is that each Euler factor in \eqref{equacao F} produces a periodic sequence of simple poles at the line $\re(s)=0$. The contribution of these simple poles to a Perron integral of the corresponding Riesz mean is roughly at most

$$\sum_{q\in S}\sum_{n=1}^{\infty}\frac{1}{n^{k}} \prod_{p\in S\setminus\{q\}} \left| 1-\exp\left(2\pi i n\frac{\log p}{\log q}\right) \right|^{-1}.$$
The convergence of this sum is, therefore, connected to the problem of how well the numbers $\log p /\log q$ can be approximated by rational numbers. 

A parameter of irrationality of an irrational number $\alpha$ is the exponent $\mu(\alpha)$ defined as the infimum over the real numbers $\eta$ such that the inequality

$$\left|\frac{a}{b}-\alpha \right|\leq \frac{C(\eta)}{b^\eta}$$
admits only a finite number of rational solutions $a/b$ with $b\geq 1$, where $C(\eta)$ is a constant that depends only on $\eta$ and the irrational number $\alpha$. A classical result of Dirichlet is that $\mu(\alpha)\geq 2$ for all irrational numbers $\alpha$, and this inequality is optimal for almost all irrational numbers, with respect to the Lebesgue measure.

In our case, the irrationality of the numbers $\log p/ \log q$ can be treated by Baker's theory of linear forms in logarithms, and we actually have by Theorem 1.1 of Bugeaud \cite{bugeaud_irrationality} (see also the references therein) the following Lemma.

\begin{lemma}\label{Lemma bugeaud} Let $p$ and $q$ be distinct prime numbers. Then there exists a constant $\gamma>0$ that does not depend on $p$ and $q$ such that, for $\mu=\gamma(\log p)(\log q)$, there is a constant $C=C(p,q)$ such that the inequality
$$\left|\frac{a}{b}-\frac{\log p}{\log q}\right|< \frac{C}{b^{\mu+1}}$$
is satisfied only for a finite number of rational numbers $a/b$, with $b\geq 1$.
\end{lemma}

\subsection{Structure of the paper} In some instances we assume that the reader is familiar with tools from Analytic Number Theory. In section \ref{secao notacao} we state the four main notations used in this paper, then we quickly proceed with proofs in section \ref{secao provas}. We then conclude with some simulations at the end.

\section{Notation}\label{secao notacao}
We use the standard notation 
\begin{enumerate}
\item $f(x)\ll g(x)$ or equivalently $f(x)=O(g(x))$;
\item $f(x)=o(g(x))$;
\item $f(x)=\Omega(g(x))$;
\item $f(x)\sim g(x)$.
\end{enumerate}
1) is used whenever there exists a constant $C>0$ such that $|f(x)|\leq C|g(x)|$, for all $x$ in a set of numbers. This set of numbers when not specified is the real interval $[L,\infty]$, for some $L>0$, but also there are instances where this set can accumulate at the right or at the left of a given real number, or at complex number. Sometimes we also employ the notation $\ll_\epsilon$ or $O_\epsilon$ to indicate that the implied constant may depends in $\epsilon$. 

In case 2), we mean that $\lim_{x}f(x)/g(x)=0$. When not specified, this limit is as $x\to \infty$ but also can be as $x$ approaches any complex number in a specific direction. 

In case 3), we say that $\limsup_x |f(x)|/|g(x)|>0$, where the limit can be taken as in the case 2).

In the last case 4), we mean that $f(x)=(1+o(1))g(x)$.

\section{proofs}\label{secao provas}

\subsection{Proof of Theorem \ref{Teorema principal}}

\begin{proof}[Proof of Theorem \ref{Teorema principal}]
Let $f$, $\chi$ and $S$ be as in the statement of Theorem \ref{Teorema principal}. Then, the Dirichlet series of $f$ can be written as follows:
\begin{align*}
F(s):=\sum_{n=1}^{\infty}\frac{f(n)}{n^s}&=\prod_{p\in S}\frac{1}{1-\frac{f(p)}{p^s}}\prod_{p\notin S}\frac{1}{1-\frac{\chi(p)}{p^s}}\\
&=\left(\prod_{p\in S}\frac{1-\frac{\chi(p)}{p^s}}{1-\frac{f(p)}{p^s}}\right)L(s,\chi).
\end{align*}

Now observe that for the primes $p\in S$ such that $f(p)=+1$, we have that $\chi(p)\neq 1$, and hence, the Euler factor corresponding to this prime in the Euler product above contributes with a simple pole at $s=0$. Similarly, for $p\in S$ with $\chi(p)=1$, we have that $f(p)\neq1$, and hence, the Euler factor corresponding to this prime in the Euler product above contributes with simple zero at $s=0$. In the other cases, the Euler factor is a regular and non-vanishing function at $s=0$.
Therefore, as $s\to 0$, we have that
$$\prod_{p\in S}\frac{1-\frac{\chi(p)}{p^s}}{1-\frac{f(p)}{p^s}}\gg \frac{1}{|s|^T},$$
where $T$ is as in Theorem \ref{Teorema principal}.

Now, recall that we define $N=\max\{0,T\}$ if $\chi(-1)=-1$, otherwise, if $\chi(-1)=1$, $N=\max\{0,T-1\}$. If $N= 0$, then there is nothing to prove since the statement of Theorem \ref{Teorema principal} becomes trivial. We therefore assume $N\geq 1$. We begin by recalling the functional equation for $L(s,\chi)$ (see for instance the book \cite{montgomerylivro}, Corollary 10.9, pg. 333):

\begin{equation}\label{equacao funcional}
L(s,\chi)=\varepsilon(\chi)L(1-s,\overline{\chi})2^s\pi^{s-1}q^{1/2-s}\Gamma(1-s)\sin\left(\frac{\pi}{2}(s+\kappa)\right),
\end{equation}
where: $|\varepsilon(\chi)|=1$, $\Gamma$ is the classical gamma function and $\kappa=0$ if $\chi(-1)=1$, $\kappa=1$ if $\chi(-1)=-1$.

By this functional equation, and the fact that $L(1,\chi)\neq 0$ for all $\chi$, we immediately see that $L(0,\chi)\neq 0$ in the situation that $\chi(-1)=-1$, and, otherwise, has a simple zero coming from the sine function at $s=0$ in the situation that $\chi(-1)=1$.

Therefore, as $s\to 0$, we have that $|F(s)|\gg \frac{1}{|s|^T}$ in the situation that $\chi(-1)=-1$, and $|F(s)|\gg \frac{1}{|s|^{T-1}}$ in the situation that $\chi(-1)=+1$.

Now, in order to complete the proof we will require the following Lemma:

\begin{lemma}\label{Lemma integral} Let $\alpha\geq 0$ and $\Gamma$ the classical Gamma function. Then, for all $\sigma>0$:
$$I(\sigma, \alpha):=\int_{1}^\infty \frac{(\log x)^\alpha}{x^{1+\sigma}}=\frac{\Gamma(\alpha+1)}{\sigma^{1+\alpha}}.$$
\end{lemma}
 
We continue with the proof of Theorem \ref{Teorema principal} and postpone the proof of the Lemma above to the end.

We recall that $F(s)$ can be written as the following Mellin transform of $\sum_{n\leq x}f(n)$ valid for $\re(s)>1$:
$$F(s)=s\int_{1}^{\infty}\left(\sum_{n\leq x}f(n)\right)\frac{dx}{x^{1+s}}.$$
By the bound \eqref{equacao cota superior}, this formula actually holds for all $\re(s)>0$. Therefore

$$\frac{|F(\sigma)|}{\sigma}\leq\int_{1}^{\infty}\left|\sum_{n\leq x}f(n)\right|\frac{dx}{x^{1+\sigma}}.$$

Now, if $\chi(-1)=-1$, we have that
$$\frac{1}{\sigma^{N+1}}\ll\int_{1}^{\infty}\left|\sum_{n\leq x}f(n)\right|\frac{dx}{x^{1+\sigma}},$$
and hence, by Lemma \ref{Lemma integral}, the partial sums $\sum_{n\leq x}f(n)$ cannot be $o((\log x)^N)$, otherwise the integral would be $o(1/\sigma^{N+1})$. Similarly, we obtain same conclusions in the case that $\chi(-1)=+1$, and this completes the proof of the Theorem. \end{proof}

\begin{proof}[Proof of Lemma \ref{Lemma integral}]
We begin by making the change $u=\log x$. Then
$$I(\sigma, \alpha)=\int_{0}^\infty \frac{u^\alpha}{e^{u\sigma}}du.$$
Now we make another substitution: $v=\sigma u$. This leads to
$$I(\sigma, \alpha)=\frac{1}{\sigma^{1+\alpha}}\int_{0}^\infty \frac{v^\alpha}{e^{v}}dv,$$
and this completes the proof.
\end{proof}

\subsection{Proof of Theorem \ref{Teorema riesz} }

\begin{proof}[Proof of Theorem \ref{Teorema riesz}] We have that (see the book \cite{montgomerylivro}, pg. 143)
$$\sum_{n\leq x}f(n)(\log x/n)^k=\frac{k!}{2\pi i}\int_{2-i\infty}^{2+i\infty}\frac{F(s)x^s}{s^{k+1}}ds.$$
Now, for $X>0$
$$\int_{2-i\infty}^{2+i\infty}\frac{F(s)x^s}{s^{k+1}}ds=\int_{2-iX}^{2+iX}\frac{F(s)x^s}{s^{k+1}}ds+O(x^2/X^k).$$
Let $R_X$ be the rectangle with vertices $2-iX, 2+iX, -1+iX$ and $-1-iX$. By formula \eqref{equacao F}, we see that $F$ has a meromorphic continuation to $\CC$ with poles only at the line $\re(s)=0$. Actually, recalling the definition of $M$ in Theorem \ref{Teorema riesz}, as we will show below, $F$ has a pole of order $\max\{0,M\}$ at $s=0$ (by a pole of order $0$ we mean without a pole), and another simple pole at this line. For a fixed $p\in S$, we have a periodic sequence of simple poles related to this prime whenever for real $t$
$$1-f(p)p^{-it}=0.$$ 
Since each $f(p)$ is a root of unity by assumption, there are coprime positive integers $1\leq a_p\leq b_p$ with $f(p)=\exp(2\pi i a_p/b_p)$. Hence, the simple poles attached to $p$ have the form
$$it_p(n):=\frac{2\pi i}{\log p}\left(n+a_p/b_p\right),\;n\in\mathbb{Z}.$$

By the Cauchy residue Theorem, $1/2\pi i$ times the integral of $k!F(s)x^s/s^{k+1}$ along $\partial R_X$, provided that $F(s)$ is regular at $s=\pm iX$, is equal to
$$\res\left(\frac{k!F(s)x^s}{s^{k+1}},s=0\right)+\sum_{p\in S}\sum_{\substack{-X\leq t_p(n)\leq X\\t_p(n)\neq 0}}\res\left(\frac{k!F(s)x^s}{s^{k+1}},s=it_p(n)\right).$$

Now we will compute the first residue at $0$ above. Observe that $1-p^{-s}\sim s\log p$ as $s\to0$. Let $a_{M+k}$ be as in Theorem \ref{Teorema riesz}. By the Euler product formula \eqref{equacao F} for $F(s)$, we have as $s\to 0$
\begin{align*}  
k!F(s)&\sim k!L(s,\chi) \left(\prod_{\substack{p\in S\\ f(p)=1}}\frac{1-\chi(p)}{s\log p}\right)\left(\prod_{\substack{p\in S\\ \chi(p)=1}}\frac{s\log p}{1-f(p)}\right)\left(\prod_{\substack{p\in S\\ f(p),\,\chi(p)\neq 1}}\frac{1-\chi(p)}{1-f(p)}\right)\\
&\sim \frac{(M+k)!a_{M+k}L(s,\chi)}{c_\chi s^T}, 
\end{align*}
where $T$ and $c_\chi$ are as in Theorem \ref{Teorema riesz}.

When $\chi(-1)=-1$, $L(s,\chi)\sim L(0,\chi)$, and when $\chi(-1)=1$, $L(s,\chi)\sim sL'(0,\chi)$. Therefore, as $s\to 0$ we have
$$F(s)\sim \frac{ a_{M+k}(M+k)!}{s^{M}}.$$

Finally, when $M+k\geq 1$, we have that $\res\left(\frac{k!F(s)x^s}{s^{k+1}},s=0\right)$ is the claimed polynomial $P(\log x)$ with degree $M+k$ and with leading coefficient $a_{M+k}$, and when $M+k\leq 0$ this residue is $O(1)$.

Now we are going to show that the contribution of the simple poles at $\re(s)=0$ is at most $O(1)$ as $x\to\infty$ in the claimed range of $k$. By formula \eqref{equacao F}, for a fixed prime $p\in S$
\begin{align*}
&\sum_{\substack{-X\leq t_p(n)\leq X\\t_p(n)\neq 0}}\res\left(\frac{F(s)x^s}{s^{k+1}},s=it_p(n)\right)\\
&\ll_p \sum_{n\in\mathbb{Z}\setminus\{0\}} \frac{1}{|n|^{k+1}} \left|L\left(it_p(n),\chi\right)\right|\prod_{\substack{q\in S\\q\neq p}}
\left|1-f(q)q^{-it_p(n)}\right|^{-1}.
\end{align*}
Now we claim that the product inside the sum above is, except for a finite number of $n$, $\ll n^{(\gamma \max_{p\in S}(\log p)^2+1)(|S|+1)} $, for some constant $\gamma>0$. 

\noindent \textit{Proof of the claim.} Observe that:
\begin{align*}
1-f(q)q^{-it_p(n)}&=1-\exp\left(2\pi i a_q/b_q-2\pi i\frac{\log q}{\log p}\left(n+a_p/b_p\right)\right)\\
&=1-\exp\left(2\pi i\left(n+a_p/b_p\right)\left( \frac{a_q}{b_q\left(n+a_p/b_p\right)}-\frac{\log q}{\log p}\right)\right)\\
&:=1-\exp(2\pi i \alpha_{p,q}(n)).
\end{align*}
Let $\|\alpha_{p,q}(n)\|$ be the distance from $\alpha_{p,q}(n)$ to the nearest integer. Then the function 
$$\frac{1-\exp(2\pi i \alpha_{p,q}(n))}{\|\alpha_{p,q}(n)\|}$$ is bounded away from $0$, and hence
$$|1-\exp(2\pi i \alpha_{p,q}(n))|^{-1}\ll \frac{1}{\|\alpha_{p,q}(n)\|}. $$
Now, by writing $\alpha_{p,q}(n)=l_{p,q}(n)+\epsilon_{p,q}(n)$, where $l_{p,q}(n)$ is an integer and $-1/2\leq \epsilon_{p,q}(n)\leq 1/2$, we see that

\begin{align*}
\|\alpha_{p,q}(n)\|=|\epsilon_{p,q}(n)|&=\left|\left(n+a_p/b_p\right)\left( \frac{a_q}{b_q\left(n+a_p/b_p\right)}-\frac{\log q}{\log p}\right)-l_{p,q}(n)\right|\\
&=\left|\left(n+a_p/b_p\right)\left( \left(\frac{a_q}{b_q}-\l_{p,q}(n)\right)\frac{1}{\left(n+a_p/b_p\right)}-\frac{\log q}{\log p}\right)\right|\\
&\gg_{p,q}\frac{1}{|n|^\mu},
\end{align*}
where the last bound holds for all but a finite number of integers $n$, due to Lemma \ref{Lemma bugeaud}, and this proves the claim.

Before we continue, we recall estimates for $L(s,\chi)$ in the $t$-aspect. The functional equation and estimates for the $\Gamma$ function give that for real $t\to\infty$ (see, for example, Corollary 10.10, pg. 334 of \cite{montgomerylivro})
$$|L(it,\chi)|\ll_{\chi} \sqrt{t}|L(1-it,\bar{\chi})|.$$

On the other hand, for non-principal $\chi$ (see Lemma 10.15 of \cite{montgomerylivro}, pg. 350), as $t\to\infty$, we have that 
$$L(1\pm it,\chi)\ll_\chi \log(|t|+1).$$

Now, going back to the contribution of the simple poles, these last estimates give that the series 
$$\sum_{n\in\mathbb{Z}\setminus\{0\}} \frac{1}{|n|^{k+1}} \left|L\left(it_p(n),\chi\right)\right|\prod_{\substack{q\in S\\q\neq p}}
\left|1-f(q)q^{-it_p(n)}\right|^{-1}$$
is absolutely convergent in the claimed range of $k$, and so, this contribution is at most $O(1)$.

To complete the proof, by combining the classical estimates for the Gamma function with \eqref{equacao F} and the functional equation \eqref{equacao funcional}, the integral
$$\int_{-1-iX}^{-1+iX}\frac{F(s)x^s}{s^{k+1}}ds\ll \int_{-1-iX}^{-1+iX}\frac{|L(s,\chi)|x^{-1}}{|s|^{k+1}}ds\ll 1/x.$$
Further, since the number of simple poles of $F(s)$ up to height $X$ is $\ll X$, there is a fixed $\delta>0$ and an infinite number of points $X_j\to\infty$ such that the distance between $\pm iX_j$ and each one of these simple poles is at least $\delta/2$. On the other hand, in the Euler product representation of $F(it)$, we have that each Euler factor corresponding to a prime $p\in S$ is $\ll |1-f(p)p^{-it}|^{-1}$, and this last function is periodic as a function of $t$, and except at the poles, it is a continuous function. Therefore, each of this Euler factors at the points $\pm iX_j$ are $O_\delta(1)$ provided that $\pm iX_j$ are $\delta/2$-distant from the poles of $(1-f(p)p^{-it})^{-1}$. Hence, along the lines $I_{\pm}=[-1\pm iX_j,2\pm iX_j]$, we have that
$$\int_{I_{\pm}}\frac{F(s)x^s}{s^{k+1}}ds\ll \frac{x^2}{X_j}.$$
Then we obtain the claimed result by making $X_j\to\infty$. \end{proof}

\section{Some simulations}\label{secao simluacoes}
Here we make some simulations with the non-principal Dirichlet character $\chi$ with modulus $3$, \textit{i.e.},
$$\chi(n)=\begin{cases}1,\mbox{ if }n\equiv 1 \mod 3,\\ -1,\mbox{ if }n\equiv 2 \mod 3,\\ 0,\mbox{ otherwise. }\end{cases}$$
For the first primes we have:

\begin{center}
\begin{tabular}{c | c| c |c |c| c| c| c| c}
$p$ &  $2$ & $3$ & $5$ & $7$ & $11$ & $13$ & $17$ & $19$\\

\hline

$\chi(p)$ & $-1$ & $0$ & $-1$ & $1$ & $-1$ & $1$ & $-1$ & $1$\\
\end{tabular}
\end{center}

In Figure \ref{figure 1}, we consider four modifications where we turn the values $\chi(p)=-1$ to $+1$, for four primes $p$. In this case, if $N$ is as in Theorem \ref{Teorema principal}, we have that $N=4$. In Figure \ref{figure 2} we consider five modifications, but $N=3$, and in Figure \ref{figure 3} we also consider four modifications, but $N=0$.

\begin{figure}[h]
\includegraphics[scale=0.5]{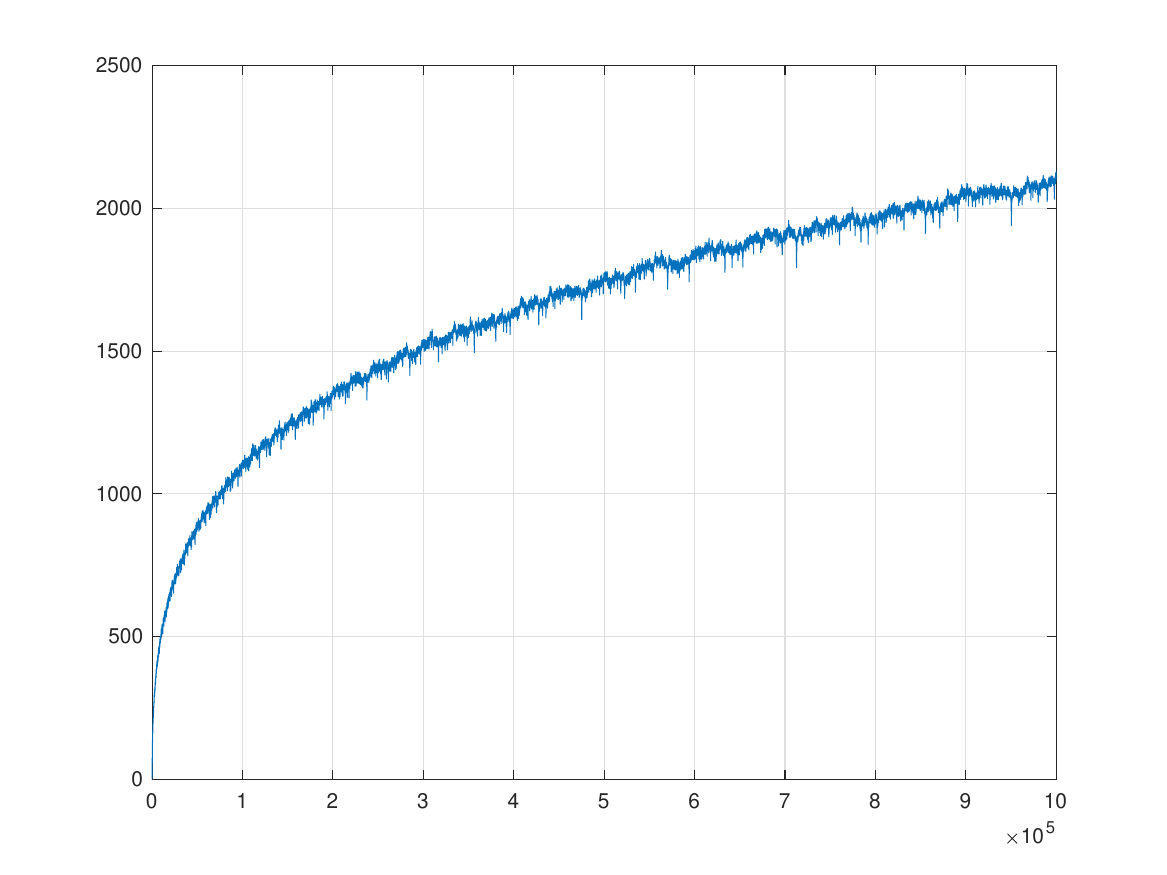}
\caption{A case where $N=4$. In blue we plot the partial sums of $f$, where $f$ is the modification of $\chi$ such that $f(2)=f(3)=f(5)=f(11)=+1$.}
\label{figure 1}
\end{figure}

\begin{figure}[h]
\includegraphics[scale=0.5]{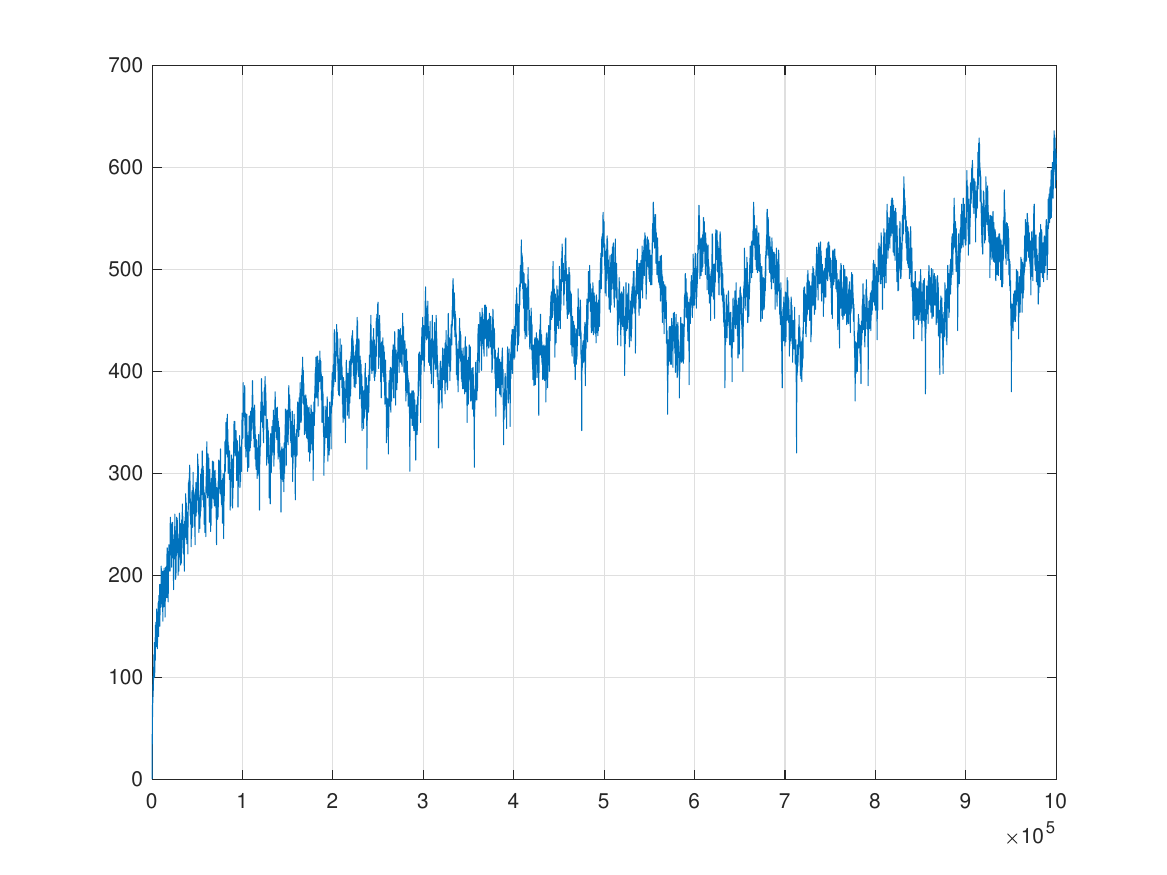} 
\caption{A case where $N=3$ with five modifications. In blue we plot the partial sums of $f$, where $f$ is the modification of $\chi$ such that $f(2)=f(3)=f(5)=f(11)=+1$ and $f(7)=-1$.}
\label{figure 2}
\end{figure}

\begin{figure}[h]
\includegraphics[scale=0.5]{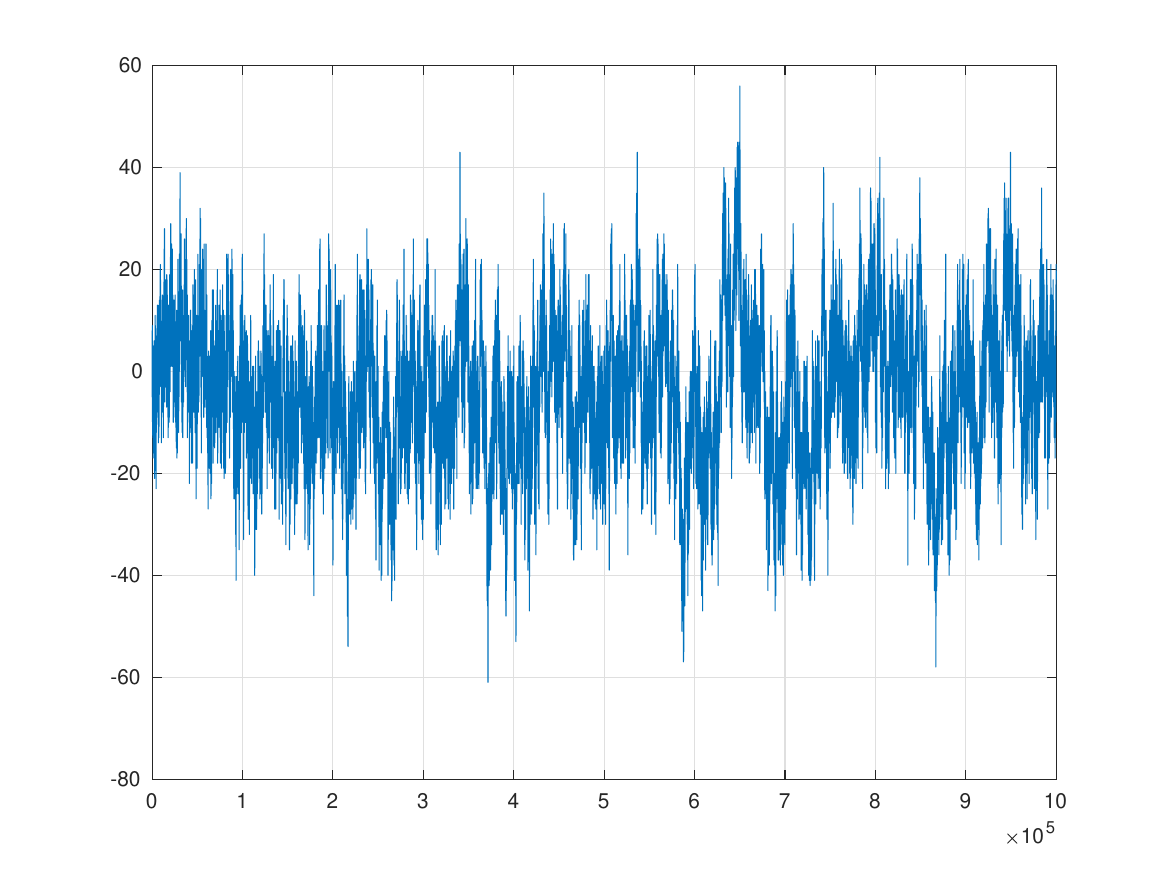} 
\caption{A case where $N=0$ with four modifications. In blue we plot the partial sums of $f$, where $f$ is the modification of $\chi$ such that $f(3)=f(7)=f(13)=f(19)=-1$.}
\label{figure 3}
\end{figure}

\newpage
\noindent \textbf{Acknowledgements}. I am warmly thankful to Carlos Gustavo Moreira (Gugu) for a fruitful discussion on Diophantine Theory and for pointing out the reference of Bugeaud \cite{bugeaud_irrationality}, and for Oleksiy Klurman for his comments on a draft version of this paper. This project was supported by CNPq grant Universal no. 403037/2021-2. The revision of this paper was made while I was a visiting professor at Aix-Marseille Université. I am thankful for their hospitality and for CNPq grant PDE no. 400010/2022-4 (200121/2022-7) for supporting this visit.

{\small{\sc \noindent
Departamento de Matem\'atica, Universidade Federal de Minas Gerais, Av. Ant\^onio Carlos, 6627, CEP 31270-901, Belo Horizonte, MG, Brazil.} \\
\textit{Email address:} aymone.marco@gmail.com}

\end{document}